\documentclass{amsart}

\renewcommand{\epsilon}{\varepsilon}

\newcommand{\carea}{c_{\mathit{area}}}

\newcommand{\Vol}{\operatorname{Vol}}
\newcommand{\CP}{{\mathbb C}\!\operatorname{P}^1}

\newcommand{\C}{{\mathbb C}}

\newcommand{\cH}{{\mathcal H}}
\newcommand{\cM}{{\mathcal M}}
\newcommand{\cQ}{{\mathcal Q}}

\newtheorem*{NNTheorem}{Theorem}
\newtheorem{Conjecture}{Conjecture}
\newtheorem{Corollary}{Corollary}

\theoremstyle{remark}
\newtheorem{Remark}{Remark}

\begin{document}

\author[A.~Aggarwal]{Amol Aggarwal}
\address{
Harvard University,
Department of Mathematics,
1 Oxford Street,
Cambridge, MA 02138,
USA
}
\email{agg\_a@math.harvard.edu}
\thanks{The research of the first author is supported by the NSF Graduate Research Fellowship under grant number DGE1144152 and NSF grant DMS-1664619}

\author[V.~Delecroix]{Vincent Delecroix}
\address{
LaBRI,
Domaine universitaire,
351 cours de la Lib\'eration, 33405 Talence, FRANCE
}
\email{20100.delecroix@gmail.com}

\author[\'E.~Goujard]{\'Elise Goujard}
\thanks{Research of the second author was partially supported by
PEPS}
\address{
Institut de Math\'ematiques de Bordeaux,
Universit\'e de Bordeaux,
351 cours de la Lib\'eration, 33405 Talence, FRANCE
}
\email{elise.goujard@gmail.com}

\author[P.~G.~Zograf]{Peter~Zograf}
\thanks{The research of Section~\ref{s:Numerical:evidence} was supported by the Russian Science Foundation grant 19-71-30002.
}
\address{
St.~Petersburg Department, Steklov Math. Institute, Fontanka 27,
St. Petersburg 191023, and Chebyshev Laboratory,
St. Petersburg State University, 14th
Line V.O. 29B, St.Petersburg 199178 Russia}
\email{zograf@pdmi.ras.ru}

\author[A.~Zorich]{Anton Zorich}
\address{
Center for Advanced Studies, Skoltech;
Institut de Math\'ematiques de Jussieu --
Paris Rive Gauche,
Case 7012,
8 Place Aur\'elie Nemours,
75205 PARIS Cedex 13, France}
\email{anton.zorich@gmail.com}

\date{December 25, 2019}

\title[Asymptotics of volumes for strata of quadratic differentials]
{Conjectural large genus asymptotics of Masur--Veech volumes
and of area Siegel--Veech constants of strata of quadratic differentials}

\begin{abstract}
We  state  conjectures  on  the asymptotic behavior of the Masur--Veech volumes of strata in the moduli  spaces  of   meromorphic quadratic  differentials  and  on the asymptotics of their  area Siegel--Veech constants  as  the genus  tends to infinity.
\end{abstract}

\maketitle

\section*{Brief history of the subject and state of the art}

The Masur--Veech volumes of strata in the moduli spaces of Abelian differentials and in the moduli spaces of meromorphic quadratic differentials with at most simple poles were introduced in fundamental papers~\cite{Masur} and~\cite{Veech}. These papers proved that the Teichm\"uller geodesic flow is ergodic on each connected component of each stratum with respect to the Masur--Veech measure introduced in these papers, and that the total measure of each stratum is finite.
\smallskip

\noindent
\textbf{Masur--Veech volumes of strata of Abelian differentials.~}
The  first  efficient evaluation of volumes of the strata in the  moduli spaces of Abelian differentials was performed by A.~Eskin and A.~Okounkov~\cite{Eskin:Okounkov} twenty years later. The Masur--Veech volumes of several low-dimensional strata of Abelian differentials were computed just before that and by different methods in~\cite{Zorich:square:tiled}.  The algorithm of A.~Eskin and A.~Okounkov was implemented by A.~Eskin in a rather efficient computer code  which  already  at  this time allowed to compute volumes of all strata of Abelian differentials up to genus 10, and volumes of some strata, like the principal one, up to genus 60 (or more).

Direct computations of volumes of the strata of Abelian differentials in genera accessible at this time combined with numerical experiments with the Lyapunov
exponents providing an approximate value of the Siegel--Veech constant $\carea$ allowed A.~Eskin and A.~Zorich to state conjectures on asymptotic behavior of volumes of the strata in the moduli spaces of Abelian differentials and on the large genus asymptotics of
the associated Siegel--Veech constants. These conjectures were stated at the end of 2003, but published in~\cite{Eskin:Zorich} only after publication of~\cite{Eskin:Kontsevich:Zorich}.
The latter paper proved the relation between the sum of the Lyapunov exponents and the Siegel--Veech constant, and, thus, justified the prior experimental evidence.

The study of the Masur--Veech volumes of the strata of Abelian differentials was tremendously advanced in recent years. D.~Chen, M.~M\"oller and D.~Zagier proved in~\cite{Chen:Moeller:Zagier} the large genus asymptotic formula for the Masur--Veech volume of the principal stratum of Abelian differentials, confirming the conjectural formula from~\cite{Eskin:Zorich} for this particular case. A.~Sauvaget proved in~\cite{Sauvaget:minimal:stratum} the asymptotics of Masur--Veech volume conjectured in~\cite{Eskin:Zorich} for the minimal stratum of Abelian differentials. A.~Aggarwal proved in~\cite{Aggarwal:Volumes} the conjectural volume asymptotics for all strata by combinatorial methods. Finally, D.~Chen, M.~M\"oller, A.~Sauavget and D.~Zagier proved in~\cite{Chen:Moeller:Sauvaget:Zagier} the conjecture in maximal generality for all connected components of all strata of Abelian differentials combining combinatorics, geometry, and intersection theory. A.~Sauavget found in~\cite{Sauvaget:expansion} the asymptotic expansion in inverse powers of $g$ for these volumes. The same authors proved in~\cite{Chen:Moeller:Sauvaget:Zagier} very efficient recursive formula for the Masur--Veech volumes of the strata of Abelian differentials, which allows to find their exact values up to very large genera. Also, A.~Aggarwal in~\cite{Aggarwal:Siegel:Veech} and D.~Chen, M.~M\"oller, A.~Sauavget and D.~Zagier in~\cite{Chen:Moeller:Sauvaget:Zagier} proved all conjectures from~\cite{Eskin:Zorich} on large genus asymptotics of the Siegel--Veech constants. Masur--Veech volumes and Siegel--Veech constants of hyperelliptic connected components were computed by J.~Athreya, A.~Eskin and A,~Zorich in~\cite{AEZ:genus:0}. Due to these series of papers we now have comprehensive information on the Masur--Veech volumes and on Siegel--Veech constants for all strata of Abelian differentials.
\smallskip

\noindent
\textbf{Masur--Veech volumes of strata of quadratic differentials.~}
Our knowledge of the Masur--Veech volumes and of the Siegel--Veech constants for the strata of meromorphic quadratic differentials is still relatively poor. Paper~\cite{AEZ:genus:0} proved exact closed formula
(conjectured by M.~Kontsevich)
for the Masur--Veech volumes and for the Siegel--Veech constants in genus zero. The first reasonable table of volumes of strata of quadratic differentials was computed by E.~Goujard in~\cite{Goujard:volumes}, where she applied the algorithm of Eskin and Okounkov~\cite{Eskin:Okounkov} to compute
volumes of all strata of dimension up to $11$.

In recent paper~\cite{DGZZ:volume} the authors found a formula for the Masur--Veech volumes of the strata with simple zeroes and simple poles through intersection numbers of $\psi$-classes on the Deligne--Mumford compactification $\overline{\cM}_{g,n}$ of the moduli space of complex curves and stated the conjecture on the large genus asymptotics of the
Masur--Veech volume of the principal stratum of holomorphic quadratic differentials.
The recent paper~\cite{Chen:Moeller:Sauvaget} suggested an alternative formula for the Masur--Veech volumes of the same strata through certain very special linear Hodge integrals. The very recent paper of~\cite{Kazarian} provided extremely efficient recursive formula for these Hodge integrals, which allows to compute volumes of strata with simple zeroes and several simple poles in genera about $100$ in split seconds. This computation
corroborates the conjectural formula of the authors for these strata. According to paper~\cite{Chen:Moeller:Sauvaget}, work in progress~\cite{Yang:Zagier:Zhang} also corroborates and
develops this conjecture. Moreover, as it will be explained in Section~\ref{s:Numerical:evidence}, formula of Chen--M\"oller--Sauvaget combined with efficient recursive algorithm of Kazarian, together allow to extend the conjecture to general strata in the moduli space of meromorphic quadratic differentials with at most simple poles.

\section{Conjectural asymptotic formula}

Let  $\boldsymbol{d}=(d_1,\dots,d_n)$  be an unordered partition of a
positive integer number $4g-4$ divisible by $4$
into a sum $|\boldsymbol{d}|=d_1+\dots+d_n=4g-4$, where $d_i\in\{-1,0,1,2,\dots\}$
 for $i=1,\dots,n$. Denote by $\hat\Pi_{4g-4}$
the set of those partitions as above, which satisfy the additional requirement that the number of entries $d_i=-1$ in $\boldsymbol{d}$ is at most $\log(g)$.

\begin{Conjecture}
\label{mc:vol}
For any $\boldsymbol{d}\in\hat\Pi_{4g-4}$ one has
\begin{equation}
\label{eq:asymptotic:formula:for:the:volume}
\Vol\cQ(d_1,\dots,d_n)=\cfrac{4}{\pi}\cdot\prod_{i=1}^n\cfrac{2^{d_i+2}}{d_i+2}
\cdot \big(1+\varepsilon_1(\boldsymbol{d})\big)\,,
\end{equation}
where
$$
\lim_{g\to\infty} \max_{\boldsymbol{d}\in\hat\Pi_{4g-4}} |\varepsilon_1(\boldsymbol{d})| = 0\,.
$$
\end{Conjecture}

\begin{Remark}
We  use  the normalization of volumes
as in~\cite{AEZ:genus:0}, \cite{Goujard:volumes},
\cite{DGZZ:volume} and~\cite{Chen:Moeller:Sauvaget}.
In   particular,   all   the   zeroes  are  labeled.
\end{Remark}

\begin{Conjecture}
\label{mc:SV}
For  non-hyperelliptic components $\cQ$ of all strata $\cQ(\boldsymbol{d})$
of meromorphic quadratic differentials with at most simple poles, where $\boldsymbol{d}\in\hat\Pi_{4g-4}$ and $g\ge 6$ one has
\begin{equation}
\label{eq:carea:to:1:4}
\carea(\cQ) = \frac{1}{4}
\cdot \big(1+\varepsilon_2(\boldsymbol{d})\big)\,,
\end{equation}
where
$$
\lim_{g\to\infty} \max_{\boldsymbol{d}\in\hat\Pi_{4g-4}} |\varepsilon_2(\boldsymbol{d})| = 0\,.
$$
\end{Conjecture}

\begin{Corollary}
Applying  the formula from~\cite{Eskin:Kontsevich:Zorich} for the sum
of  the Lyapunov exponents of the Hodge bundle over the Teichm\"uller
geodesic  flow  in  any  nonhyperelliptic  connected  component  of a
stratum $\cQ(\boldsymbol{d})$ of meromorphic quadratic differentials
with $\boldsymbol{d}\in\hat\Pi(4g-4)$
we get
\begin{equation}
\label{eq:sum:of:plus:exponents}
\lambda^+_1 + \dots + \lambda^+_g
\ = \
\cfrac{1}{24}\cdot\sum_{d_i\in \boldsymbol{d}} \cfrac{d_i(d_i+4)}{d_i+2}
\ +\ \frac{\pi^2}{12}+\frac{\pi^2}{3}\cdot\epsilon_2(\boldsymbol{d})\,.
\end{equation}
\end{Corollary}

\begin{Remark}
Paper~\cite{Chen:Moeller:Sauvaget} proved formulae for the Masur--Veech volume of the principal strata of meromorphic quadratic differentials
and for the corresponding Siegel--Veech constant $\carea$ in terms of intersection numbers. This paper also suggested analogous conjectural formula for other strata. We hope that the algebro-geometric conjectures from~\cite{Chen:Moeller:Sauvaget} and the numerical conjectures as above might be mutually useful.
\end{Remark}

\section{Numerical  evidence.}
\label{s:Numerical:evidence}
\noindent
\textbf{Siegel--Veech constant.~}
Numerical evaluation of Lyapunov exponents allows to compute approximate values of Siegel--Veech constants $\carea(\cQ)$ for connected components $\cQ$ of strata $\cQ(\boldsymbol{d})$ in the moduli space of meromorphic quadratic differentials with at most simple poles through the following formula from~\cite{Eskin:Kontsevich:Zorich}:
$$
\lambda^-_1 + \dots + \lambda^-_{g_{\mathit{eff}}}
\ = \
\cfrac{1}{24}\cdot\sum_{d_i\in \boldsymbol{d}} \cfrac{d_i(d_i+4)}{d_i+2}
\ +\
\cfrac{1}{4}\cdot\sum_{\substack{d_j\in \boldsymbol{d}\\d_j\equiv 1\mathit{mod}\, 2}}
\cfrac{1}{d_j+2}
\ +\
\frac{\pi^2}{3}\cdot\carea(\cQ)\,.
$$
Here the \textit{effective genus} $g_{\mathit{eff}}$ is defined as $g_{\mathit{eff}}=\hat g - g$, where $\hat g$ is the genus of the canonical double cover $p: \hat C\to C$ (ramified at simple poles and at zeroes of odd orders of the quadratic differential $q$) such that
$p^\ast q=\hat\omega$ is a square of a globally defined Abelian differential $\hat\omega$ on $\hat C$.

Simulations performed for numerous strata in large genera provide strong numerical evidence for conjectural asymptotics~\eqref{eq:carea:to:1:4} from Conjecture~\ref{mc:SV}. For example, experiments performed for about $20$ random strata with $0,1,2,3$ simple poles in genera in the range from $20$ to $30$ give approximate values of $\carea(\cQ)$ varying from $0.247$ to $0.259$ for $10^6$ iterations of fast Rauzy induction. Analogous experiments with the principal stratum $\cQ_{g,0}=\cQ(1^{4g-4})$ give
approximate values of
$\carea(\cQ_{g,0})$ varying from $0.251$ to $0.253$ for $10^7$ iterations of fast Rauzy induction for genera from $20$ to $30$ and
approximate values of
$\carea(\cQ_{g,0})$ varying from $0.255$ to $0.258$ for $10^6$ iterations of fast Rauzy induction for genera about $40$. Similar simulations for the principal strata were independently performed by Ch.~Fougeron.

As we explained in the introduction,
combining very recent results~\cite{Chen:Moeller:Sauvaget} and~\cite{Kazarian} we can compute
exact values of volumes of strata $\cQ_{g,n}=\cQ(1^{4g-4+n},-1^n)$
and then compute exact values of the Siegel--Veech constants $\carea\big(\cQ_{g,n}\big)$ due to the following result.

\begin{NNTheorem}[\cite{Goujard:carea}]
Let $g$ be a strictly positive integer, and $n$ nonnegative integer.
When $g=1$ we assume that $n\ge 2$. Under the above conventions
the following formula is valid:
\begin{align}
\label{eq:carea:Elise}
\carea(\cQ_{g,n}) =
\frac{1}{\Vol(\cQ_{g,n})}\cdot\Bigg(
\frac{1}{8}
 \sum_{\substack{g_1+g_2=g
 \\ n_1+n_2=n+2 \\
 g_i\geq 0, n_i\geq 1, d_i\geq 1}}
 \frac{\ell!}{\ell_1!\ell_2!}\frac{n!}{(n_1-1)!(n_2-1)!}\cdot
 \\
\notag
\cdot\frac{(d_1-1)!(d_2-1)!}{(d-1)!}
\Vol(\cQ_{g_1, n_1})\times\Vol(\cQ_{g_2,n_2})
\ +\hspace*{80pt}
\\+\
\notag
\frac{1}{16}\cdot\frac{(4g-4+n)n(n-1)}{(6g-7+2n)(6g-8+2n)}
\Vol(\cQ_{0, 3})\times\Vol(\cQ_{g,n-1})\ +
\\
\notag
+\ \frac{\ell!}{(\ell-2)!}\frac{(d-3)!}{(d-1)!}
 \Vol(\cQ_{g-1,n+2})
\Bigg)
 \,.
\end{align}
Here $d=\dim_{\C}\cQ_{g,n}=6g-6+2n$, $d_i=6g_i-6+2n_i$,
$\ell=4g-4+n$, $\ell_i=4g_i-4+n_i$\,.
\end{NNTheorem}

The above formula gives exact values of $\carea(Q_{g,n})$ for pairs $g,n$ up to $g=250$ and larger. The analysis of the resulting data provides serious numerical evidence towards validity of Conjecture~\ref{mc:SV}. In particular, for $g=3,4,\dots,250$, the Siegel--Veech constant $\carea(Q_{g,0})$
is monotonously decreasing from $0.284275$ to $0.250285$.
For a fixed genus $g$ and small
values $n=0,1,2,\dots$ of $n$, the Siegel--Veech constant
$\carea(Q_{g,n})$ does not
fluctuate much. For example, the approximate values of $\carea(Q_{250,n})$ for $n=0,1,\dots,6$ are given by the following list:
$$
\{
0.250285, 0.250118, 0.249992, 0.249909, 0.249867, 0.249867, 0.249909
\}\,.
$$

\smallskip

\noindent
\textbf{Volume asymptotics.~}
The arguments towards Conjecture~\ref{mc:vol} are indirect. We start by recalling the scheme which was used to formulate analogous Conjecture in~\cite{Eskin:Zorich} for large genus asymptotics of Masur--Veech volumes of strata
of Abelian differentials.
Certain phenomena which had a status of conjecturs at the early stage of the project~\cite{Eskin:Zorich}
are proved in the recent paper~\cite{Aggarwal:Siegel:Veech} of A.~Aggarwal.
We conjecture that strata in the moduli spaces of meromorphic quadratic differentials have analogous geometric properties.

Having explained the scheme in the case of Abelian differentials (where everything is proved by now) we describe necessary adjustments which we use for the case of quadratic differentials. In this latter case, part of results are still conjectural.
\smallskip

\noindent
\textbf{Stating volume asymptotics conjecture for Abelian differentials.~}
Paper~\cite{Eskin:Masur:Zorich}
provided a formula for
the Siegel--Veech constant $\carea(\cH(m))$
of any (connected component of any) stratum $\cH(m)$
in the moduli space of Abelian differentials. This formula expressed $\carea(\cH(m))$
through Masur--Veech volumes of the associated \textit{principal boundary} strata normalized by $\Vol\cH(m)$. It was conjectured that in large genera a dominant part of the contribution to $\carea(\cH(m))$ comes from the boundary strata corresponding to simplest degenerations (represented by so-called \textit{configurations of multiplicity one}).
Neglecting contributions of more sophisticated boundary strata one obtains particularly simple approximate formula
for $\carea(\cH(m))$ as an explicit weighted sum of ratios
of volumes of simplest boundary strata over the volume of the stratum $\cH(m)$.
The only way to get the constant asymptotic value $\carea(\cH(m))\approx\tfrac{1}{2}$ in this expression leads to the formula
\begin{equation}
\label{eq:up:to:const}
\Vol\cH(m_1,\dots,m_n)= const\cdot\prod_{i=1}^n\cfrac{1}{m_i+1}
\cdot \big(1+\varepsilon_3(m)\big)\,,
\end{equation}
where $\varepsilon_3(m)\to 0$ as $g\to+\infty$ uniformly
for all partitions $m$ of $2g-2$ into a sum of positive integers.

Evaluation of the universal constant $const$ in the above formula is a separate nontrivial problem. Since the Masur--Veech volumes of strata appear in the approximate formula
for $\carea(\cH(m))$ only in ratios of the volume of some smaller stratum over the volume of the original stratum, the global normalization constant cancels out in every ratio. In the case of Abelian differentials, the constant $const$ was guessed from numerics.
Namely, in the particular case of the principal stratum,
the original formula of A.~Eskin and A.~Okounkov from~\cite{Eskin:Okounkov} allows to compute $\Vol\cH(1^{2g-2})$ for genus $g=60$
and higher, which allowed to guess the correct value
$const=4$. Asymptotics~\eqref{eq:up:to:const}
was rigorously proved for particular cases in~\cite{Chen:Moeller:Zagier} and in~\cite{Sauvaget:minimal:stratum} and then in~\cite{Aggarwal:Volumes}
and~\cite{Chen:Moeller:Sauvaget:Zagier} in general case.

We address the reader interested in more details to~\cite{Zorich:Appendix:2:Amol} where the notions of \textit{configuration of multiplicity one} is explained in details, and where the contributions of such configurations to all possible Siegel--Veech constants, including $\carea(\cH(m))$, are computed in full details. The conjecture that the total contribution to $\carea(\cH(m))$ coming from more complicated configurations becomes negligible in large genera is recently proved by A.~Aggarwal in~\cite{Aggarwal:Siegel:Veech}.
\smallskip

\noindent
\textbf{Stating volume asymptotics conjecture for quadratic differentials.~}
Developing technique from~\cite{Eskin:Masur:Zorich} and using the topological description of the principal boundary strata in the moduli spaces of quadratic differentials given in~\cite{Masur:Zorich},
E.~Goujard obtained in~\cite{Goujard:carea} an expression for the Siegel--Veech constant $\carea(\cQ(\boldsymbol{d}))$
through Masur--Veech volumes of the stratum $\cQ(\boldsymbol{d})$ and of its principal boundary. Formula~\eqref{eq:carea:Elise}
is a particular case of this more general formula.

Analogously to the situation with Abelian differentials, we expect that the total contribution
to $\carea(\cQ(\boldsymbol{d}))$ coming from the principal boundary strata different from the simplest ones becomes negligible in high genera. We describe this conjecture in more details in Section~\ref{s:linearised:carea}. Extracting from the formula of Goujard for $\carea(\cQ(\boldsymbol{d}))$ the contribution
coming from these simplest degenerations we get an approximate formula for $\carea(\cQ(\boldsymbol{d}))$
through a weighted sum of ratios of volumes of the boundary strata divided by the volume of the stratum under consideration. In the particular case of the strata $\cQ_{g,n}$ this corresponds to removing
from~\eqref{eq:carea:Elise} the summands containing products of the volumes,
see Section~\ref{s:linearised:carea} for more details.

Similarly to the case of Abelian differentials, (up to a global normalization factor) expression~\eqref{eq:asymptotic:formula:for:the:volume} is the unique expression, such that for any $\boldsymbol{d}\in\hat\Pi$
the resulting weighted sum tends to the asymptotic
value $\carea(\cQ(\boldsymbol{d}))\approx \frac{1}{4}$.
We provide this weighted sum for $\carea(\cQ(\boldsymbol{d}))$ in Section~\ref{s:linearised:carea} below.

As in the case of Abelian differentials, this approach does not allow to find the global constant factor in the asymptotic formula~\eqref{eq:asymptotic:formula:for:the:volume}.
The corresponding factor $\tfrac{4}{\pi}$ was originally conjectured in~\cite{DGZZ:volume} for the large genus asymptotics of the Masur--Veech volume of the principal stratum of holomorphic quadratic differentials. This conjecture is based on fine
geometric considerations combined with elaborate computations
performed in~\cite{DGZZ:volume}.

The formula from~\cite{Chen:Moeller:Sauvaget} for the same volume combined with a very efficient recursion from~\cite{Kazarian} for the corresponding linear Hodge integrals, involved in this formula, provide together very serious numerical evidence
for validity of our conjecture for the strata with only simple zeroes and with simple poles, when the number of simple poles is small enough.

For any pair $g,n$ of nonnegative integers define
$$
\Vol^{appr}\cQ_{g,n}=\cfrac{4}{\pi}\cdot 2^n\cdot\left(\frac{8}{3}\right)^{4g-4+n}\,.
$$
This value corresponds to the right-hand side of expression~\eqref{eq:asymptotic:formula:for:the:volume}
applied to the partition $\boldsymbol{d}=(1^{4g-4+n},-1^n)$, where
the factor $\big(1+\varepsilon_1(\boldsymbol{d})\big)$ is omitted.
The sequence of ratios $\cfrac{\Vol\cQ^{appr}_{g,0}}{\Vol\cQ_{g,0}}$ is monotonously decreasing for the range $g=3,4,\dots,250$, from $\cfrac{\Vol\cQ^{appr}_{3,0}}{\Vol\cQ_{3,0}}\approx 1.01892$ to $\cfrac{\Vol\cQ^{appr}_{250,0}}{\Vol\cQ_{250,0}}\approx 1.00027$. For any fixed genus $g$ and small
values $n=0,1,2,\dots$ of $n$, the volumes do not fluctuate much. For example, the approximate values of $\cfrac{\Vol\cQ^{appr}_{250,n}}{\Vol\cQ_{250,n}}$ for $n=0,1,\dots,5$ are given by the following list:
$$
\{
{1.00027406, 1.00027477, 1.00027493, 1.00027481, 1.00027469, 1.00027484}\}\,.
$$
In particular, these numerical data corroborates our conjectural value $\tfrac{4}{\pi}$ of the universal normalizing constant.

\begin{Remark}
Analogous analysis of numerical data was independently performed by D.~Chen, M.~M\"oller and A.~Sauvaget who used the algorithm
from~\cite{Yang:Zagier:Zhang}.
\end{Remark}
\smallskip

\noindent
\textbf{Range of validity of the conjecture.~}
E.~Goujard computed exact values of the Masur--Veech volumes of strata in the moduli space of quadratic differentials for strata of dimension at most $11$.
Consider expression~\eqref{eq:asymptotic:formula:for:the:volume}
with omitted factor  $\big(1+\varepsilon_1(\boldsymbol{d})\big)$
and divide it
by the exact value of the volume $\Vol\cQ(\boldsymbol{d})$. The resulting ratio
evaluated for strata of genera $5$ and $6$
of dimension $11$ varies from
$1.037$ to $1.135$, which suggests that
formula~\eqref{eq:asymptotic:formula:for:the:volume}
gives reasonable approximation of the Masur--Veech volume already for strata of relatively small genus.

\section{Contribution of configurations of multiplicity one to $\carea(\cQ(\boldsymbol{d}))$.}
\label{s:linearised:carea}

We refer the reader to~\cite{Masur:Zorich} and to~\cite{Goujard:volumes} for the notions  ``\textit{configurations of homologous saddle connections}'' and ``\textit{principal boundary strata}''.

\begin{Conjecture}
\label{conj:config}
For all strata $\cQ(d_1, \dots, d_n)$ with $(d_1,\dots,d_n)\in\hat\Pi(4g-4)$,
the Siegel--Veech constant  $\carea(\cQ(d_1, \dots, d_n))$ is asymptotically
supported on the following set of configurations:

$$
\begin{array}{rl}
\includegraphics{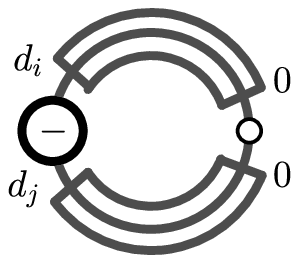}
\vspace*{60pt}
&
\includegraphics{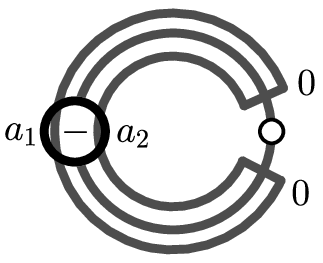}
\\
\mathcal{C}_{b,\mathrm{I}}(d_i,d_j)
\hspace*{40pt}
&
\hspace*{40pt}
\mathcal{C}_{b,\mathrm{II}}(a_1, a_2)
\end{array}
$$
where $d_i,d_j, a_1, a_2$ are
such that $d_i, d_j\geq 1$, $\{d_i,d_j\}\subset\{d_1, \dots, d_n\}$, where $i\neq j$, $a_1, a_2\geq 0$, $a_1+a_2\ge 3$ and $a_1+a_2+2\in\{d_1, \dots, d_n\}$.
\end{Conjecture}

\begin{Corollary}
\label{cor:1:and:3:imply:2}
For $n\le\log(g-1)-2$
a combination of
Conjecture~\ref{mc:vol} and Conjecture~\ref{conj:config} implies Conjecture~\ref{mc:SV}.
\end{Corollary}
\begin{proof}
First note that our normalization of the Masur--Veech volume implies that
$\Vol\cQ(0,d_1,\dots,d_n)=2\Vol\cQ(d_1,\dots,d_n)$,
see~\cite{AEZ:genus:0}. Hence,
$$
\varepsilon_1(0,d_1,\dots,d_n)=\varepsilon_1(d_1,\dots,d_n)\,,
$$
where the quantity $\varepsilon_1(\boldsymbol{d})$ is defined in~\eqref{eq:asymptotic:formula:for:the:volume}.
Thus, it is sufficient to prove Corollary~\ref{cor:1:and:3:imply:2} for partitions which do not contain entries $0$ (representing the marked points).

Choose any pair of indices $1\le i< j\le n$ such that $d_i,d_j\ge 1$.
Applying Theorem~1 from~\cite{Goujard:carea}, we
obtain the following expression for the contribution
$c_{area}(\mathcal{C}_{b,\mathrm{I}}(d_i,d_j))$
of the configuration $\mathcal{C}_{b,\mathrm{I}}(d_i,d_j)$
to $\carea(\cQ(\boldsymbol{d}))$:
\[c_{area}(\mathcal{C}_{b,\mathrm{I}}(d_i,d_j))=\frac{(d-3)!}{(d-1)!}\cdot 2d_id_j\cdot \frac{\Vol\cQ_{g-1}(d_1, \dots, d_i-2, \dots, d_j-2, \dots, d_n)}{\Vol\cQ_{g}(d_1, \dots, d_n)}\,,\]
where $d=\dim_{\C}\cQ_g(d_1, \dots, d_n)=2g+n-2$.
Note that the symbol $\boldsymbol{d}$ in bold denotes a partition in $\hat\Pi(3g-3+n)$, while the symbol $d$
denotes the complex dimension of the stratum $\cQ(\boldsymbol{d})$.

Assuming that $g\gg 1$ and using Conjecture~\ref{mc:vol} we obtain
\begin{equation}
\label{eq:config:I}
c_{area}(\mathcal{C}_{b,\mathrm{I}}(d_i,d_j))\sim \frac{(d_i+2)(d_j+2)}{2^3 d^2}
\,.
\end{equation}

Choose any index $i$ in $\{1,\dots,n\}$ such that $d_i\ge 3$ (if such index exists). By
Theorem~1 from~\cite{Goujard:carea},
for any pair of nonnegative integers $a_1, a_2$ satisfying $a_1+a_2=d_i-2$, the contribution
$c_{area}(\mathcal{C}_{b,\mathrm{II}}(a_1,a_2))$
of the configuration
$\mathcal{C}_{b,\mathrm{II}}(a_1,a_2)$
to $\carea(\cQ(\boldsymbol{d}))$
has the following form:
\[c_{area}(\mathcal{C}_{b,\mathrm{II}}(a_1,a_2))=\frac{(d-3)!}{(d-1)!}
\cdot 2(a_1+a_2)\cdot\frac{\Vol\cQ_{g-1}(d_1, \dots, d_i-4, \dots d_n)}{\Vol\cQ_g(d_1, \dots, d_n)}\,,\]
so, assuming that $g\gg 1$ and
using Conjecture~\ref{mc:vol}, we obtain the following contribution to
$\carea(\cQ(\boldsymbol{d}))$ of all pairs $a_1,a_2$
corresponding to the fixed $d_i$ as above:
\begin{multline}
\label{eq:config:II}
\frac{1}{2}\sum_{a_1=0}^{d_i-2}
c_{area}(\mathcal{C}_{b,\mathrm{II}}(a_1,a_2))
\frac{1}{2^4 d^2}(d_i-1)(d_i+2)
=\\=
\frac{1}{2^4 d^2}\big((d_i+2)^2-3(d_i+2)\big)
\,.
\end{multline}

We use notation $\cQ(d_1, \dots d_n)=\cQ(-1^{\mu_{-1}}, 1^{\mu_1}, 2^{\mu_2}, \dots)$ to record the multiplicities of entries $-1,1,2,\dots$ in the partition $\boldsymbol{d}$.

Equation~\eqref{eq:config:I} implies that
\begin{multline}
\label{eq:tmp1}
2^3 d^2
\sum_{\substack{1\le i< j\le n\\d_i,d_j\ge 1}}
c_{area}(\mathcal{C}_{b,\mathrm{I}}(d_i,d_j))
\sim
\sum_{1\le i< j\le n}(d_i+2)(d_j+2)
\ -\\-\ \mu_{-1}\sum_{i=1}^n (d_i+2)
+\frac{\mu_{-1}(\mu_{-1}+1)}{2}\,.
\end{multline}

Equation~\eqref{eq:config:II} implies that
\begin{multline}
\label{eq:tmp2}
2^4 d^2
\sum_{\substack{1\le i\le n\\d_i\ge 3}}\sum_{\substack{a_1+a_2=d_i-2\\a_1,a_2\ge 0}}
c_{area}(\mathcal{C}_{b,\mathrm{II}}(a_1,a_2))
\sim
\sum_{i=1}^n (d_i+2)^2 - 3\sum_{i=1}^n (d_i+2)
\ -\\-\ (\mu_{-1}+9\mu_1+16\mu_2)
+3(\mu_{-1}+3\mu_1+4\mu_2)
\,.
\end{multline}

Note that
\begin{equation}
\label{eq:tmp3}
\sum_{i=1}^n (d_i+2) =  \left(\sum_{i=1}^n d_i\right) + 2n = 4g-4+2n=2d\,,
\end{equation}
and, in particular, $\mu_{-1}+\mu_1+\mu_2+\dots=n<d$,
which implies that
\begin{equation}
\label{eq:tmp4}
\frac{2\mu_{-1}-6\mu_1-12\mu_2}{d^2}\to 0\quad\text{as }d\to+\infty\,.
\end{equation}
Note also, that by assumption the
number of simple poles in a stratum is
much smaller than the genus $g$. Since $g<d$ we get
$$
\frac{\mu_{-1}}{d}\to 0\quad\text{as }d\to+\infty\,.
$$

We conclude that for partitions $\boldsymbol{d}$
as in the statement of Corollary~\ref{cor:1:and:3:imply:2}
with $\mu_0=0$ we get the following asymptotics as $g\to+\infty$:
\begin{multline*}
c_{area}(\cQ(\boldsymbol{d}))
\sim
\sum_{\substack{1\le i< j\le n\\d_i,d_j\ge 1}}
c_{area}(\mathcal{C}_{b,\mathrm{I}}(d_i,d_j))
+
\sum_{\substack{1\le i\le n\\d_i\ge 3}}\sum_{\substack{a_1+a_2=d_i-2\\a_1,a_2\ge 0}}
c_{area}(\mathcal{C}_{b,\mathrm{II}}(a_1,a_2))
\ \sim\\ \sim\
\frac{1}{4}\frac{1}{(2d)^2}\left(
\left(\sum_{i=1}^n (d_i+2)\right)^2
-3\sum_{i=1}^n (d_i+2)\right)
 =
\frac{1}{4}\frac{1}{(2d)^2}\big( (2d)^2 - 3\cdot 2d\big)
\sim \frac{1}{4}\,,
\end{multline*}
where the first equivalence is the statement of Conjecture~\ref{conj:config} and the second equivalence
is the combination of~\eqref{eq:tmp1}--\eqref{eq:tmp4}.
\end{proof}

\begin{Remark}
Note that in this proof we used much weaker restriction on the growth rate
of the number of simple poles, namely it was sufficient to have $\mu_{-1}=o(g)$
as $g\to+\infty$.
\end{Remark}

\subsection*{Acknowledgements}
We thank D.~Chen, M.~M\"oller
and A.~Sauvaget for their formula and for their interest to the conjectural large genus asymptotics described above. We are very much indebted to M.~Kazarian for his very efficient recursive formula for the linear Hodge integrals providing a reliable test of our conjectures.



\begin{thebibliography}{CM\"oZa}

\bibitem[Ag1]{Aggarwal:Volumes}
A.~Aggarwal,
\textit{Large Genus Asymptotics for Volumes of Strata of Abelian Differentials}, to appear in J. Amer. Math. Soc.,
\texttt{arXiv:1804.05431}.

\bibitem[Ag2]{Aggarwal:Siegel:Veech}
A.~Aggarwal,
\textit{Large Genus Asymptotics for Siegel--Veech Constants},
to appear in GAFA, \texttt{arXiv:1810.05227}.

\bibitem[AEZ2]{AEZ:genus:0}
J.~Athreya, A.~Eskin, and A.~Zorich,
\textit{Right-angled  billiards  and  volumes  of  moduli  spaces  of
quadratic differentials on $\CP$},  Ann. Scient. ENS,
4\`eme s\'erie, \textbf{49} (2016),  1307--1381.

\bibitem[CM\"oZa]{Chen:Moeller:Zagier}
D.~Chen, M.~M\"oller, D.~Zagier,
\textit{Quasimodularity and large genus limits of Siegel--Veech constants},
J. Amer. Math. Soc. \textbf{31} (2018), no. 4, 1059--1163.

\bibitem[CM\"oSZa]{Chen:Moeller:Sauvaget:Zagier}
D.~Chen, M.~M\"oller, A.~Sauvaget, D.~Zagier,
\textit{Masur--Veech volumes and intersection theory on moduli spaces of abelian differentials};
\texttt{arXiv:1901.01785} (2019).

\bibitem[CM\"oS]{Chen:Moeller:Sauvaget}
D.~Chen, M.~M\"oller, A.~Sauvaget
with an appendix by G. Borot, A. Giacchetto, D.~Lewanski,
\textit{Masur--Veech volumes and intersection theory:
the principal strata of quadratic differentials},
\texttt{arXiv:1912.02267} (2019).

\bibitem[DGZZ]{DGZZ:volume}
V. Delecroix, E. Goujard, P. Zograf, A. Zorich,
\textit{Masur-Veech volumes, frequencies of simple closed geodesics and intersection numbers of moduli spaces of curves},
\texttt{arXiv:1908.08611} (2019).

\bibitem[EKZo]{Eskin:Kontsevich:Zorich}
A.~Eskin, M.~Kontsevich, A.~Zorich,
\textit{  Sum  of Lyapunov exponents of the Hodge bundle with respect
to the Teichm\"uller geodesic flow},
Publications de l'IHES, \textbf{120:1} (2014),
207--333.

\bibitem[EMaZo]{Eskin:Masur:Zorich}
A.~Eskin, H.~Masur, A.~Zorich,
\textit{Moduli  Spaces  of  Abelian  Differentials:  The  Principal
Boundary, Counting  Problems,  and  the Siegel--Veech Constants},
Publications de l'IHES, {\bf 97:1} (2003), pp. 61--179.

\bibitem[EO]{Eskin:Okounkov}
A.~Eskin, A.~Okounkov,
\textit{Asymptotics  of number of  branched coverings of a torus and
volumes   of   moduli   spaces  of  holomorphic   differentials},
Invent. Math., {\bf 145:1} (2001), 59--104.

\bibitem[EZo]{Eskin:Zorich}
A.~Eskin, A.~Zorich,
\textit{Volumes of strata of Abelian differentials and
Siegel--Veech constants in large genera},
Arnold Mathematical Journal,
\textbf{1:4} (2015), 481--488.

\bibitem[G1]{Goujard:carea}
E.~Goujard,
\textit{Siegel--Veech constants for strata of moduli spaces of quadratic differentials},
GAFA, \textbf{25:5} (2015), 1440--1492.

\bibitem[G2]{Goujard:volumes}
E.~Goujard,
\textit{Volumes of strata of moduli spaces of quadratic differentials: getting explicit values},
Ann. Inst. Fourier,
\textbf{66} no. 6 (2016), 2203--2251.

\bibitem[G2tab]{Goujard:table}
E.~Goujard,
\textit{Table of volumes of strata of quadratic differentials (up to dimension
 $11$)},
\texttt{https://sites.google.com/site/elisegoujard/home/recherche-research/}\\
\texttt{bhtablevoljune2015.pdf}.

\bibitem[Kaz]{Kazarian}
M.~Kazarian,
\textit{Recursion for Masur--Veech volumes of moduli
spaces of quadratic differentials},
\texttt{arXiv:1912.10422} (2019).

\bibitem[Ma]{Masur}
H.~Masur,
\textit{Interval exchange transformations and measured  foliations},
Annals of Math., {\bf 115} (1982), 169--200.

\bibitem[MaZo]{Masur:Zorich}
H.~Masur, A.~Zorich,
\textit{Multiple Saddle  Connections on  Flat  Surfaces and  Principal
Boundary of the Moduli Spaces of Quadratic Differentials}, GAFA,
\textbf{18} (2008) 919--987.

\bibitem[S1]{Sauvaget:minimal:stratum}
A.~Sauvaget,
\textit{Volumes and Siegel--Veech constants of
$\mathcal{H}(2g-2)$ and Hodge integrals},
GAFA \textbf{28} (2018), no. 6, 1756--1779.

\bibitem[S2]{Sauvaget:expansion}
A.~Sauvaget,
\textit{The Large genus asymptotic expansion
of Masur-Veech volumes}, (2019)
\texttt{arXiv:1903.04454}, to appear in IMRN.

\bibitem[V]{Veech}
W.~Veech,
\textit{Gauss measures for transformations on the space of interval
exchange maps}, Annals of Math.,
\textbf{115} (1982), 201--242.

\bibitem[YZZ]{Yang:Zagier:Zhang}
D.~Yang, D.~Zagier, and Y.~Zhang,
\textit{Asymptotics of the Masur-Veech
Volumes}, (2019), in preparation.

\bibitem[Zo1]{Zorich:square:tiled}
A.~Zorich,
\textit{Square tiled surfaces and Teichm\"uller volumes of the moduli spaces of abelian differentials.}
Rigidity in dynamics and geometry (Cambridge, 2000), 459--471,
Springer, Berlin, 2002.

\bibitem[Zo2]{Zorich:Appendix:2:Amol}
A.~Zorich,
\textit{Asymptotic values of Siegel–Veech constants}, Appendix to the paper of A.~Aggarwal
\textit{Large Genus Asymptotics for Volumes of Strata of Abelian Differentials}, to appear in J.~Amer.~Math.~Soc., \texttt{arXiv:1804.05431}.

\end{thebibliography}
\end{document}